\documentclass[12pt, reqno]{amsart}
\usepackage{amsmath, amsthm, amscd, amsfonts, amssymb, graphicx, color}
\usepackage[bookmarksnumbered, colorlinks, plainpages]{hyperref}
\hypersetup{colorlinks=true,linkcolor=red, anchorcolor=green, citecolor=cyan, urlcolor=red, filecolor=magenta, pdftoolbar=true}

\newtheorem{theorem}{Theorem}[section]
\newtheorem{lemma}[theorem]{Lemma}

\theoremstyle{definition}

\newtheorem{example}[theorem]{Example}

\theoremstyle{remark}
\newtheorem{remark}[theorem]{\bf{Remark}}
\numberwithin{equation}{section}
\begin{document}
	
\title [New semi-norm of semi-Hilbertian space operators   ]{\small {   New semi-norm of semi-Hilbertian space operators and its application   } }

\author[P. Bhunia, A. Sen and K. Paul] {Pintu Bhunia, Anirban Sen  and Kallol Paul}

\address{(Bhunia) Department of Mathematics, Jadavpur University, Kolkata 700032, West Bengal, India}
\email{pintubhunia5206@gmail.com; 	pbhunia.math.rs@jadavpuruniversity.in}

\address [Sen] {Department of Mathematics, Jadavpur University, Kolkata 700032, West Bengal, India}
\email{anirbansenfulia@gmail.com}

\address{(Paul) Department of Mathematics, Jadavpur University, Kolkata 700032, West Bengal, India}
\email{kalloldada@gmail.com; kallol.paul@jadavpuruniversity.in}

\thanks{First author would like to thank  CSIR, Govt. of India for the financial support in the form of JRF. Second  author would like to thank UGC, Govt. of India for the financial support in the form of SRF}


\subjclass[2010]{Primary 47A12, Secondary  47A30, 47A63.}
\keywords{A-numerical radius; A-operator seminorm; Semi-Hilbertian space; Inequality}

\maketitle

\begin{abstract}
In this paper, we introduce a new semi-norm of operators on a semi-Hilbertian space, which generalizes the A-numerical radius and A-operator semi-norm. We study the basic properties of this semi-norm, including upper and lower bounds for it. As an application of this new semi-norm, we obtain upper bounds for the A-numerical radius of semi-Hilbertian space operators, which are sharper than the earlier ones.
\end{abstract}

\section{\textbf{Introduction}}

  \noindent  Let $\mathcal{B}(\mathcal{H})$ denote the $C^*$-algebra of all bounded linear operators on a complex Hilbert space $\mathcal{H}$ with usual inner product $\langle ., . \rangle$  and $\|.\|$ be the corresponding norm on $\mathcal{H}$, induced by the usual inner product  $\langle ., .\rangle.$ Throughout this paper,  $O$ is the zero operator on $\mathcal{H}$ and $I$ is the identity operator on $\mathcal{H}.$  For $T \in \mathcal{B}(\mathcal{H}),$ $R(T)$ and $N(T)$ denote the range space of $T$ and null space of $T$, respectively. We denote, the norm closure of $R(T)$ in $\mathcal{H}$ by  $\overline{R(T)}$ and the orthogonal projection onto $\overline{R(T)}$  by $P_{T}.$ Let $A$ be a self-adjoint operator in $\mathcal{B}(\mathcal{H}).$ Then $A$ is called positive if $\langle Ax,x \rangle \geq 0$ for all $x \in \mathcal{H}$. Let $\mathcal{B}(\mathcal{H})^+$ denote the cone of all positive operators on $\mathcal{B}(\mathcal{H}),$ i.e., $$\mathcal{B}(\mathcal{H})^+=\left\{ A \in \mathcal{B}(\mathcal{H})~:~\langle Ax,x \rangle \geq 0 ~\forall x \in \mathcal{H} \right\}.$$
Henceforth we reserve the symbol $A$   for a positive operator on $\mathcal{H}$, i.e.,  $A \in \mathcal{B}(\mathcal{H})^+ .$  A functional $\langle . , . \rangle_A$ ~:~ $\mathcal{H} \times \mathcal{H} ~\longrightarrow \mathbb{C}$ defined by $\langle x, y \rangle_A=\langle Ax, y \rangle$ for all $x,y \in \mathcal{H}$, is a semi-inner  product on  $\mathcal{H}$ and correspondingly $\left(\mathcal{H}, \langle . , .\rangle_A\right)$ is a semi-inner product space. The semi-norm induced by the semi-inner  product $\langle ., . \rangle_A$ is denoted by $\|.\|_{A},$ i.e., $\|x\|_{A}=\sqrt{\langle x , x \rangle_A}=\|A^{1/2}x\|$ for all $x \in \mathcal{H}.$ Clearly, $\|x\|_{A}=0$ if and only if $ x \in N(A)$ and therefore, $\|.\|_A$ is a norm if and only if $A$ is  one-one. Also, the semi-normed space $\left(\mathcal{H}, \| . \|_A\right)$ is complete if $R(A)$ is closed in  $\mathcal{H}.$
     For $T \in \mathcal{B}(\mathcal{H}),$ the $A$-operator semi-norm of $T$ is given by
    \begin{align*}
    	\|T\|_A &= \sup \left\{\frac{\|Tx\|_A}{\|x\|_A}~:~x \in \overline{R(A)}, x \neq 0\right\}.
    \end{align*}
   Now, $\|T\|_A < \infty$ if there exists a constant $c>0$ such that $\|Tx\|_{A} \leq c\|x\|_{A}$ for all $x \in \mathcal{H}$ and in this case, $T$ is called $A$-bounded operator. 
   For $T \in \mathcal{B}(\mathcal{H}),$ $A$-minimum modulus of $T$, denoted by $m_{A}(T)$ and is defined by  
     $$m_{A}(T)=  \inf \left\{\frac{\|Tx\|_A}{\|x\|_A}~:~x \in \overline{R(A)}, x \neq 0\right\}.$$
     For $T \in \mathcal{B}(\mathcal{H}),$ an operator $R \in \mathcal{B}(\mathcal{H})$ is an $A$-adjoint of $T$ if $\langle Tx , y \rangle_A = \langle x , Ry \rangle_A$ for every $x,y \in \mathcal{H}$.  For any $T \in \mathcal{B}(\mathcal{H}),$ the existence of an $A$-adjoint of $T$ is not guaranteed. In fact, an operator $T \in \mathcal{B}(\mathcal{H})$ may admit none, one or many $A$-adjoint. By Douglas Theorem (\cite{dog}) we have that for an operator $T \in \mathcal{B}(\mathcal{H})$ if $R(T^*A) \subseteq R(A)$, then $T$ admits an $A$-adjoint. 
     Let  $\mathcal{B}_{A}(\mathcal{H})$ denote the collection of all operators in $\mathcal{B}(\mathcal{H})$ which admit $A$-adjoint, i.e., 
     $$\mathcal{B}_A(\mathcal{H})=\{T\in \mathcal{B}(\mathcal{H}) :R(T^*A) \subseteq R(A) \}.$$ For $T\in \mathcal{B}_A(\mathcal{H})$, the operator equation $AX=T^*A$ has a unique solution $T^{\sharp_A}$, satisfying $R(T^{\sharp_A})\subseteq \overline{R(A)}$.
     Here $T^*$ is the adjoint of $T.$  Clearly, for $T \in \mathcal{B}_{A}(\mathcal{H})$, $AT^{\sharp_A}=T^*A.$ Note that $T^{\sharp_A}= A^{\dagger}T^*A,$ where $A^{\dagger}$ is the Moore-Penrose inverse of A, (see \cite{ms}). Again, by applying Douglas theorem, we have, $\mathcal{B}_{A^{1/2}}(\mathcal{H})$ is the set of all operators in $\mathcal{B}(\mathcal{H})$ which admit $A^{1/2}$ adjoint, i.e., $R(T^*A^{1/2})\subseteq R(A^{1/2})$. It is well-known 
     that $$\mathcal{B}_{A^{1/2}}(\mathcal{H})=\{T \in \mathcal{B}(\mathcal{H})~:~\exists~~c>0~~\textit{such that } ~~ \|Tx\|_A\leq c\|x\|_A~~ \forall x\in \mathcal{H}\}.$$ Therefore, the operators in  $\mathcal{B}_{A^{1/2}}(\mathcal{H})$ are called $A$-bounded operators. Notice that  $\mathcal{B}_{A}(\mathcal{H})$ and  $\mathcal{B}_{A^{1/2}}(\mathcal{H})$ are sub-algebras of  $\mathcal{B}(\mathcal{H})$.
    Moreover, $\mathcal{B}_{A}(\mathcal{H}) \subseteq \mathcal{B}_{A^{1/2}}(\mathcal{H}) \subseteq \mathcal{B}(\mathcal{H})$ and these inclusions become equalities if $A$ is one-one and has closed range.
    An operator $T \in \mathcal{B}(\mathcal{H})$ is said to be $A$-self-adjoint if $AT$ is self-adjoint, i.e., $AT=T^*A.$ If $T$ is A-self-adjoint, then  $T \in \mathcal{B}_{A}(\mathcal{H})$ and in general, $T\neq T^{\sharp_A}$. For  $T \in \mathcal{B}_A(\mathcal{H})$, $T= T^{\sharp_A}$ if and only if $T$ is $A$-self-adjoint and $R(T) \subseteq \overline{R(A)}.$  
   For $T, S \in \mathcal{B}_{A}(\mathcal{H})$, we have $((T^{\sharp_A})^{\sharp_A})^{\sharp_A}= T^{\sharp_A},
   ~(T^{\sharp_A})^{\sharp_A}= P_ATP_A, 
   ~(TS)^{\sharp_A}=S^{\sharp_A}T^{\sharp_A}.$ 
   Also, for $T, S \in \mathcal{B}_{A^{1/2}}(\mathcal{H})$, we have $~\|TS\|_A \leq \|T\|_A \|S\|_A,
   ~\|Tx\|_A \leq \|T\|_A \|x\|_A\, ~\mbox{for all $x$ in $\mathcal{H}$}.$
    Clearly,  $T^{\sharp_A}T$ and $TT^{\sharp_A}$ are $A$-self-adjoint operators and satisfy $\|TT^{\sharp_A}\|_A = \|T^{\sharp_A}T\|_A = \|T^{\sharp_A}\|^2_A = \|T\|^2_A.$
     For $T \in \mathcal{B}_{A^{1/2}}(\mathcal{H}),$ the $A$-numerical
     range  and the A-numerical radius of $T$, denoted by $W_{A}(T)$ and $w_A(T)$, are defined respectively as
     $$W_{A}(T)=\{ \langle Tx,x \rangle_{A} ~:~  x \in \mathcal{H}, \|x\|_{A}=1 \}.$$
     and 
      $$w_{A}(T)=\sup \{ |\langle Tx,x \rangle_{A}| ~:~  x \in \mathcal{H}, \|x\|_{A}=1 \}.$$
    Similar as the A-numerical radius, the A-Crawford number of $T\in \mathcal{B}_{A^{1/2}}(\mathcal{H})$, denoted as $c_A(T)$ and is defined as 
    $$c_{A}(T)=\inf \{ |\langle Tx,x \rangle_{A}| ~:~  x \in \mathcal{H}, \|x\|_{A}=1 \}.$$
    It is well known that the A-numerical radius and the A-operator semi-norm are equivalent semi-norm on $\mathcal{B}_{A^{1/2}}(\mathcal{H})$, satisfying that
    for every $T \in \mathcal{B}_{A^{1/2}}(\mathcal{H}),$  
    $$\frac{1}{2}\|T\|_{A} \leq w_{A}(T) \leq \|T\|_{A}.$$
    The numerical radius inequalities  and the $A$-numerical radius inequalities 
   have been studied by many mathematicians over the years, we refer the readers to see \cite{ 3, 1, 2, 4, zama} for some classical numerical radius inequalities and \cite{ a1, a2, feki 2, mose, zamani} for some  $A$-numerical radius inequalities. Motivated by the new norm  \cite{SBBP}, for each $\alpha \in [0,1],$   we introduce $A_\alpha$-semi-norm on $\mathcal{B}_{A^{1/2}}(\mathcal{H}),$ defined as follows: 
	$$ \left\|T\right\|_{A_\alpha}~=~\sup \left\{ \sqrt{\alpha|\langle Tx,x\rangle_{A}|^2+(1-\alpha) \|Tx\|_{A}^2}~:~ x\in \mathcal{H},\|x\|_{A}=1 \right\}.$$
	 If $\alpha=1,$ then $\left\|T\right\|_{A_\alpha}=w_{A}(T)$ and if $\alpha=0,$ then $\left\|T\right\|_{A_\alpha}=\|T\|_{A}.$ Considering $A=I$ we get, recently introduced the $\alpha$-norm  on $\mathcal{B}(\mathcal{H})$,  (see in \cite{BBSP,SBBP}).\\
	 In this article, we study the basic properties of this semi-norm. We also develope upper and lower bounds for $A_\alpha$-numerical radius. Applying these inequalities we obtain upper bounds for the $A$-numerical radius which improve and generalize the existing ones.

	\section{\textbf{Main results}}
	
\noindent 	We begin this section with the following theorem, proof of which follows easily from the definition of $A_\alpha$-semi-norm.

	\begin{theorem}\label{theo1}
		$\left\|~.~\right\|_{A_\alpha}$ defines a semi-norm on  $\mathcal{B}_{A^{1/2}}(\mathcal{H}).$ This semi-norm is equivalent to the A-numerical radius and the A-operator semi-norm, satisfying the following inequalities:
		\begin{eqnarray*}
		(i)~~	&&w_{A}(T) \leq \left\|T\right\|_{A_\alpha} \leq \sqrt{(4-3\alpha )}w_{A}(T),\\
		(ii)~~ &&  \max\left\{\frac{1}{2},\sqrt {1-\alpha} \right\}\|T\|_{A} \leq \left\|T\right\|_{A_\alpha} \leq \|T\|_{A}.
		\end{eqnarray*}	
	\end{theorem}

Next we prove the following equivalent conditions.

    \begin{theorem}\label{th11}
    	Let $T \in \mathcal{B}_{A^{1/2}}(\mathcal{H})$ and let $\alpha \neq 0,1.$ Then the following conditions are equivalent:\\
    	(i)~$\left\|T\right\|_{A_\alpha}=\sqrt{\alpha w_{A}^2(T)+(1-\alpha) \|T\|^2_{A}}.$\\
    	(ii)~  There exists a sequence $\{x_{n}\} $ in $ \mathcal{H}$ with $\|x_{n}\|_{A}=1$ such that 
    	  	  $\lim_{n\to\infty}\|Tx_{n}\|_{A}$ $=\|T\|_{A}$ and $\lim_{n\to\infty}|\langle Tx_{n},x_{n}\rangle_{A}|=w_{A}(T).$
        \end{theorem}
     \begin{proof}
        (i) $\Longrightarrow$ (ii):
     	Clearly from the definition of $\|T\|_{A_\alpha}$ it follows that there exists a sequence  $\{x_{n}\} $ in $\mathcal{H}$ with $\|x_{n}\|_{A}=1$ such that $$\left\|T\right\|_{A_\alpha}=\lim_{n\to\infty} \sqrt{ \alpha|\langle Tx_{n},x_{n}\rangle_{A}|^2+(1-\alpha) \|Tx_{n}\|_{A}^2 }. $$
     	Now, $\{|\langle Tx_{n},x_{n}\rangle_{A}|\}$ and $\{\|Tx_{n}\|_{A}\}$ are both bounded sequences of real numbers and so there exists a subsequence $\{x_{n_{k}}\}$ of  $\{x_{n}\}$ such that $\{|\langle Tx_{n_{k}},x_{n_{k}}\rangle_{A}|\}$ and $\{\|Tx_{n_{k}}\|_{A}\}$ are convergent. Therefore, we get 
     	\begin{align*}
     		\alpha w_{A}^2(T)+(1-\alpha) \|T\|^2_{A}&=\left\|T\right\|^2_{A_\alpha}\\&= \lim_{k\to\infty}\left\{  \alpha|\langle Tx_{n_{k}},x_{n_{k}}\rangle_{A}|^2+(1-\alpha) \|Tx_{n_{k}}\|_{A}^2 \right\} \\ & =\alpha \lim_{k\to\infty} |\langle Tx_{n_{k}},x_{n_{k}}\rangle_{A}|^2 + (1-\alpha) \lim_{k\to\infty} \|Tx_{n_{k}}\|_{A}^2\\ & \leq	\alpha w_{A}^2(T)+(1-\alpha)  \|T\|^2_{A}.
        \end{align*}
       This implies that
       	$\lim_{k\to\infty}\|Tx_{n_{k}}\|_{A}=\|T\|_{A}$ and $\lim_{k\to\infty}|\langle Tx_{n_{k}},x_{n_{k}}\rangle_{A}|=w_{A}(T).$
       	
       	(ii) $\Longrightarrow$ (i):
       From the definition of $A_\alpha$-semi-norm we have 
       	\begin{align*}
       	\left\|T\right\|_{A_\alpha} & =  \sup_{\|x\|_{A}=1} \sqrt{ \alpha|\langle Tx,x\rangle_{A}|^2+(1-\alpha) \|Tx\|_{A}^2 }\\ & \geq \lim_{n\to\infty} \sqrt{ \alpha|\langle Tx_{n},x_{n}\rangle_{A}|^2+(1-\alpha) \|Tx_{n}\|_{A}^2 } \\ & = \sqrt{\alpha w_{A}^2(T)+(1-\alpha) \|T\|^2_{A}}.
       	\end{align*}
       This completes the proof.
     \end{proof}

Next, we study an important basic property for the $A_\alpha$-semi-norm. First we recall that
an operator $U\in  \mathcal{B}_A(\mathcal{H})$ is said to be $A$-unitary if $\|Ux\|_A=\|U^{\sharp_A}x\|_A=\|x\|_A$, for all $x\in \mathcal{H}$.
It was shown in \cite{acg1} that an operator $U\in  \mathcal{B}_A(\mathcal{H})$ is $A$-unitary if and only if $U^{\sharp_A} U=(U^{\sharp_A})^{\sharp_A} U^{\sharp_A}=P_A$.

\begin{theorem}
	If $T \in \mathcal{B}_{A}(\mathcal{H})$, then
	\[\|UTU^{\sharp_A}\|_{A_\alpha}=\|T\|_{A_\alpha}\,\, \text{for every A-unitary operator }\,\, U.\]
\end{theorem}

\begin{proof}
	For $S\in  \mathcal{B}_{A}(\mathcal{H}),$ we consider a set $L_S$ as follows:
	\[L_S= \left \{  \left (\sqrt{\alpha}\langle Sx,x\rangle_A , \sqrt{1-\alpha}\|Sx\|_A \right ) : x\in \mathcal{H}, \|x\|_A=1 \right \}.\]
	Let $(\lambda, \mu)\in L_{UTU^{\sharp_A}}$. Then there exists $x\in \mathcal{H}$ with $\|x\|_A=1$ such that $\lambda=\sqrt{\alpha}\langle UTU^{\sharp_A}x,x\rangle_A$ and $\mu=\sqrt{1-\alpha}\|UTU^{\sharp_A}x\|_A$. It is easy to verify that $\lambda=\sqrt{\alpha}\langle TU^{\sharp_A}x,U^{\sharp_A}x\rangle_A$ and $\mu=\sqrt{1-\alpha}\|TU^{\sharp_A}x\|_A$. Since $\|U^{\sharp_A}x\|_A=\|x\|_A$ for all $x\in \mathcal{H}$, so $(\lambda, \mu)\in L_{T}$. Therefore, we have $ L_{UTU^{\sharp_A}}\subseteq L_{T}.$ \\
	Now, let $(\beta,\gamma)\in L_{T}$. Then, there exist $x\in \mathcal{H}$ with $\|x\|_A=1$ such that 
	$\beta=\sqrt{\alpha}\langle Tx,x\rangle_A $ and $\gamma= \sqrt{1-\alpha}\|Tx\|_A$, i.e., $\gamma= \sqrt{(1-\alpha)  \langle UTx,UTx\rangle_A}$. Let $x=P_Ax+y$ where $y\in N(A)$. Since $N(A)$ is invariant subspace under $T$, i.e., $T(N(A))\subseteq N(A)$, so by using this argument we get,
	$$\beta=\sqrt{\alpha}\langle TP_Ax,P_Ax\rangle_A\,\,\text{and}\,\, \gamma= \sqrt{(1-\alpha)  \langle UTP_Ax,UTP_Ax\rangle_A}.$$
	Since $U^{\sharp_A}U=P_A$, so we have, 
		$\beta=\sqrt{\alpha}\langle UTU^{\sharp_A}Ux,Ux\rangle_A\,\,\text{and}\,\, $ 
		$$\gamma= \sqrt{(1-\alpha)  \langle UTU^{\sharp_A}Ux,UTU^{\sharp_A}Ux\rangle_A}=\sqrt{1-\alpha}\| UTU^{\sharp_A}Ux\|_A.$$
		Now $\|Ux\|_A=\|x\|_A$  for all $x\in \mathcal{H}$, so  $(\beta,\gamma)\in L_{UTU^{\sharp_A}}.$ Therefore, we have $L_T\subseteq L_{UTU^{\sharp_A}}$. Hence, we conclude that $L_T= L_{UTU^{\sharp_A}}$. This completes the proof.
	
\end{proof}

 In the following theorem we obtain a lower bound for the $A_\alpha$-semi-norm.

	\begin{theorem}\label{theo2}
		Let $T\in \mathcal{B}_{A}(\mathcal{H}). $ Then
		\begin{align*}
			 \left\|T\right\|_{A_\alpha}^2  \geq \max \Big\{  \alpha w_{A}^2(T)+(1-\alpha) c_{A}(T^{\sharp_A}T), \alpha c_{A}^2(T)+(1-\alpha) \|T\|_{A}^2,\\ 2\sqrt{\alpha (1-\alpha) }w_{A}(T)\sqrt{c_{A}(T^{\sharp_A}T)}, 2\sqrt{\alpha (1-\alpha) }c_{A}(T)\|T\|_{A} \Big\}.
		\end{align*}
    \end{theorem}

    \begin{proof}
    	Let $x \in \mathcal{H}$ with $\|x\|_{A}=1$. Then we get,
    	\begin{eqnarray*}
    		\left\|T\right\|_{A_\alpha}^2 && \geq \alpha|\langle Tx,x\rangle_{A}|^2+(1-\alpha) \|Tx\|_{A}^2 \\ && = \alpha|\langle Tx,x\rangle_{A}|^2 + (1-\alpha) \langle T^{\sharp_A}Tx,x\rangle_{A} \\ && \geq \alpha|\langle Tx,x\rangle_{A}|^2 + (1-\alpha) c_{A}(T^{\sharp_A}T).
    	\end{eqnarray*}	
        Taking supremum over all $x$ in $\mathcal{H}$ with $\|x\|_{A}=1,$ we get that
        \begin{eqnarray}\label{eqn1}
              \left\|T\right\|_{A_\alpha}^2 \geq  \alpha w_{A}^2(T)+(1-\alpha)c_{A}(T^{\sharp_A}T).
        \end{eqnarray}
       Proceeding similarly we also get
        \begin{eqnarray}\label{eqn2}
        \left\|T\right\|_{A_\alpha}^2 \geq \alpha c_{A}^2(T)+(1-\alpha) \|T\|_{A}^2.
        \end{eqnarray}
        Now, we also have
        \begin{eqnarray*}
        	\left\|T\right\|_{A_\alpha}^2 && \geq \alpha|\langle Tx,x\rangle_{A}|^2+(1-\alpha)\|Tx\|_{A}^2 \\  && \geq 2\sqrt{\alpha (1-\alpha) }|\langle Tx,x\rangle_{A}|\|Tx\|_{A} \\ && \geq  2\sqrt{\alpha (1-\alpha) }|\langle Tx,x\rangle_{A}|\sqrt{c_{A}(T^{\sharp_A}T)}.
        \end{eqnarray*}
     Taking supremum over all $x$ in $\mathcal{H}$ with $\|x\|_{A}=1$, we get
       \begin{eqnarray}\label{eqn3}
       \left\|T\right\|_{A_\alpha}^2 \geq  2\sqrt{\alpha (1-\alpha)  }w_{A}(T)\sqrt{c_{A}(T^{\sharp_A}T)}.
       \end{eqnarray}
       Similarly,
        \begin{eqnarray}\label{eqn4}
        \left\|T\right\|_{A_\alpha}^2 \geq  2\sqrt{\alpha (1-\alpha)  }c_{A}(T)\|T\|_{A}.
        \end{eqnarray}
        Combining (\ref{eqn1}), (\ref{eqn2}), (\ref{eqn3}) and (\ref{eqn4}) we get our desired inequality.
      \end{proof}
	
	Next, we obtain upper bounds for the $A_\alpha$-semi-norm of the product of two operators.
	
	\begin{theorem}\label{theo3}
		Let $T,S \in \mathcal{B}_{A^{1/2}}(\mathcal{H})$ and $\alpha \neq 1.$ Then 
		$$\left\|TS\right\|_{A_\alpha} \leq { \min\left\{ \frac{2}{\sqrt{1-\alpha}},\frac{1}{(1-\alpha)},{4} \right\} }\left\|T\right\|_{A_\alpha}\left\|S\right\|_{A_\alpha}.$$
	\end{theorem}

    \begin{proof}
    		From the definition of $A_\alpha$-semi-norm we get,
    	\begin{eqnarray*}
     	\left\|TS\right\|_{A_\alpha}^2 && = \sup_{\|x\|_{A}=1} \left\{ \alpha|\langle TSx,x\rangle_{A}|^2+\beta \|TSx\|_{A}^2 \right\} \\
     	 && \leq \alpha 
     	w_{A}^2(TS) + (1-\alpha) \| TS \|_{A}^2 \\
     	 && \leq \alpha \| TS \|_{A}^2 + (1-\alpha)\| TS \|_{A}^2 \\ 
     	 && \leq \| T \|_{A}^2 \| S \|_{A}^2 \\ 
     	 && \leq 4 w_{A}^2(T)\| S \|_{A}^2,~~\mbox{as $\| T \|_{A} \leq 2w_{A}(T)$ } \\
     	  && \leq  4 w_{A}^2(T)\left(1/(1-\alpha)\right)\left\|S\right\|_{A_\alpha}^2,~~\mbox{by Theorem \ref{theo1}} \\
     	   && \leq \frac{4}{1-\alpha}\left\|T\right\|_{A_\alpha}^2\left\|S\right\|_{A_\alpha}^2,~~\mbox{by Theorem \ref{theo1}}.
     \end{eqnarray*}
     	Thus, 
     \begin{eqnarray}\label{p1}
     		\left\|TS\right\|_{A_\alpha} \leq \sqrt{\frac{4}{1-\alpha}}\left\|T\right\|_{A_\alpha}\left\|S\right\|_{A_\alpha}.
     \end{eqnarray}
      Proceeding as above, and using $\sqrt{1-\alpha}\left\|T\right\|_{A} \leq \left\|T\right\|_{A_\alpha}$ and $\sqrt{1-\alpha}\left\|S\right\|_{A} \leq \left\|S\right\|_{A_\alpha}$ we get,
      \begin{eqnarray}\label{p2}
      \left\|TS\right\|_{A_\alpha} \leq \sqrt{\frac{1}{(1-\alpha)^2}}\left\|T\right\|_{A_\alpha}\left\|S\right\|_{A_\alpha}.
      \end{eqnarray}
      Finally, using $\left\|T\right\|_{A} \leq 2w_{A}(T)$ and $\left\|S\right\|_{A} \leq 2w_{A}(S)$, and by Theorem \ref{theo1} we get,
      \begin{eqnarray}\label{p3}
      \left\|TS\right\|_{A_\alpha} \leq 4 \left\|T\right\|_{A_\alpha}\left\|S\right\|_{A_\alpha}.
       \end{eqnarray}
      Combining (\ref{p1}),(\ref{p2}) and (\ref{p3}) we get the required inequality.
    \end{proof}
	
	To prove next inequality we need the following lemma.
	
	\begin{lemma}\label{lemma1}
		Let $T,S \in \mathcal{B}_{A}(\mathcal{H}).$ Then the following inequalities hold:\\ 
		    (i)~~If $TS=ST,$ then $w_{A}(TS) \leq 2w_{A}(T)w_{A}(S).$
			\\ (ii)~~If $(TS)^{\sharp_A}=T^{\sharp_A}S,$ then $w_{A}(TS) \leq w_{A}(T)\left\|S\right\|_{A}.$
	\end{lemma}
	\begin{proof}
	(i)~ Without loss of generality we assume that $w_{A}(T)=1$ and  $w_{A}(S)=1.$  Now, we have that
		\begin{align*}
			w_{A}(TS) &= w_{A}\left( \frac{1}{4}\left((T+S)^2-(T-S)^2\right)\right)\\
			& \leq \frac{1}{4}\left(w_{A}((T+S)^2)+w_{A}((T-S)^2)\right)\\
			& \leq \frac{1}{4}\left(w_{A}^2(T+S)+w_{A}^2(T-S)\right),~\mbox{as $w_{A}(T^n) \leq w_{A}^{n}(T)~~ \forall n \in \mathbb{N}$} \\ & \leq \frac{1}{4}\left( (w_{A}(T)+w_{A}(S))^2+(w_{A}(T)+w_{A}(S))^2\right) \\&= 2.
		\end{align*}
	   This completes the proof.\\
	   (ii)~See in \cite[Cor. 3.3]{zamani}.
	\end{proof}

		\begin{theorem}\label{theo4}
			Let $T,S \in \mathcal{B}_{A}(\mathcal{H}).$ Then the following inequalities hold:\\
		       (i)~~If $TS=ST$ and $\alpha \neq 1,$ then  $$\left\|TS\right\|_{A_\alpha} \leq \sqrt{\left({4\alpha}+\frac{1}{1-\alpha}\right)} \left\|T\right\|_{A_\alpha}\left\|S\right\|_{A_\alpha}.$$
		       (ii)~~If $(TS)^{\sharp_A}=T^{\sharp_A}S$ and $\alpha \neq 1,$ then  $$\left\|TS\right\|_{A_\alpha} \leq \sqrt{{1+\alpha}}\min\left\{ {2},\frac{1}{\sqrt{1-\alpha}} \right\} \left\|T\right\|_{A_\alpha}\left\|S\right\|_{A_\alpha}.$$
			\end{theorem}
	
	    \begin{proof}
	    	(i)~From the definition of $A_\alpha$-semi-norm we get, 
	    	\begin{align*}
	    		\left\|TS\right\|^2_{A_\alpha}  & =\sup_{\|x\|_{A}=1} \left\{ \alpha|\langle TSx,x\rangle_{A}|^2+(1-\alpha) \|TSx\|_{A}^2 \right\}\\ &  \leq \alpha w_{A}^2(TS) + (1-\alpha) \left\|TS\right\|_{A}^2 \\ & \leq  4\alpha w_{A}^2(T)w_{A}^2(S)+ (1-\alpha) \left\|T\right\|_{A}^2\left\|S\right\|_{A}^2, \mbox{by Lemma \ref{lemma1}(i)}\\ & \leq \left({4\alpha}+\frac{1}{1-\alpha}\right)\left\|T\right\|^2_{A_\alpha}\left\|S\right\|^2_{A_\alpha},~~\mbox{by Theorem \ref{theo1}}.
	        \end{align*}
            Therefore, 
            $$\left\|TS\right\|_{A_\alpha} \leq \sqrt{\left({4\alpha} +\frac{1}{1-\alpha}\right)} \left\|T\right\|_{A_\alpha}\left\|S\right\|_{A_\alpha}.$$
        	(ii)~Using the definition of $A_\alpha$-semi-norm we get,
        	\begin{align*}
        		\left\|TS\right\|^2_{A_\alpha}  & = \sup_{\|x\|_{A}=1} \left\{ \alpha|\langle TSx,x\rangle_{A}|^2+(1-\alpha) \|TSx\|_{A}^2 \right\}\\ & \leq \alpha w_{A}^2(TS) + (1-\alpha) \left\|TS\right\|^2_{A} \\ & \leq \alpha w_{A}^2(T)\left\|S\right\|^2_{A}+ (1-\alpha) \left\|T\right\|^2_{A}\left\|S\right\|^2_{A}, \mbox{by Lemma \ref{lemma1}(ii)}\\ & \leq  {\alpha} \left\|T\right\|^2_{A_\alpha}\left\|S\right\|^2_{A}+\left\|T\right\|^2_{A_\alpha}\left\|S\right\|^2_{A},~~\mbox{by Theorem \ref{theo1}}\\ & = \left({\alpha} +1\right)\left\|T\right\|^2_{A_\alpha}\left\|S\right\|^2_{A}\\ & \leq \left({\alpha}+1\right)\min\left\{ {2},\frac{1}{\sqrt{1-\alpha}} \right\}^2\left\|T\right\|^2_{A_\alpha}\left\|S\right\|^2_{A_\alpha},~~\mbox{by Theorem \ref{theo1}}.
        	\end{align*}
          Therefore,
           $$\left\|TS\right\|_{A_\alpha} \leq \sqrt{\left( {\alpha}+1\right)}\min\left\{ {2},\frac{1}{\sqrt{1-\alpha}} \right\} \left\|T\right\|_{A_\alpha}\left\|S\right\|_{A_\alpha}.$$
	    \end{proof}

   \noindent 
   
   Next we need the following two  notations: For  $T \in \mathcal{B}_{A}(\mathcal{H}),$ let $\Re_{A}(T)=\frac{1}{2}(T+T^{\sharp_{A}})$ and  $\Im_{A}(T)=\frac{1}{2\rm i}(T-T^{\sharp_{A}})$. Then, $T$ can be expressed as $T=\Re_A(T)+\rm i \Im_{A}(T)$. By using this decomposition we prove the following inequalities.
    
    \begin{theorem}\label{theo5}
    		Let $T \in \mathcal{B}_{A}(\mathcal{H}).$ Then the following inequalities  hold :
    			\begin{align*}
    			&(i)~~ \left\|T\right\|_{A_\alpha}^2 \geq \max	\Bigg\{\frac{1}{2}\left\|\frac{\alpha}{4}(T^{\sharp_A}T+TT^{\sharp_A})+(1-\alpha) T^{\sharp_A}T\right\|_{A},\\
    			& \,\,\,\,\,\,\,\,\,\,\,\,\,\,\,\,\,\,\,\,\,\,\,\,\,\,\,\,\,\,\,\,\,\,\,\,\,\,\,\,\,\,\,\,\,\,\,\,\,\,\,\,\,\,\,\,\,\,\,\,\,\,\,\,\,\,\,\,\,\,\,\,\,\,\,\,\,\,\,\,\,\,\,\,\,\,\,\,\,\,\,\,\,\,\,\,\,\,\, \frac{1}{3}\left\|\frac{\alpha}{2}(T^{\sharp_A}T+TT^{\sharp_A})+(1-\alpha) T^{\sharp_A}T\right\|_{A}\Bigg\},\\
    		  	&(ii)~~ \left\|T\right\|_{A_\alpha}^2 \leq  \left\|\frac{\alpha}{2}(T^{\sharp_A}T+TT^{\sharp_A})+(1-\alpha) T^{\sharp_{A}}T\right\|_{A},\\
      		&(iii)~~ w_{A}^2(T) \leq   \min_{\alpha} \left\|\frac{\alpha}{2}(T^{\sharp_A}T+TT^{\sharp_A})+(1-\alpha) T^{\sharp_A}T\right\|_{A} \leq  \frac{1}{2}\left\|T^{\sharp_A}T+TT^{\sharp_A}\right\|_{A}.
      	\end{align*}
    \end{theorem}

    \begin{proof}
    	We observe that 
    	 $(\Re_{A} (T^{\sharp_A}))^{\sharp_A}=\Re_{A} (T^{\sharp_A})$ and $(\Im_{A} (T^{\sharp_A}))^{\sharp_A}=\Im_{A} (T^{\sharp_A}),$ and  
    	\begin{align*}
    	\Re_{A}^2 (T^{\sharp_A})+\Im_{A}^2 (T^{\sharp_A})=\frac{T^{\sharp_A}(T^{\sharp_A})^{\sharp_A}+(T^{\sharp_A})^{\sharp_A}T^{\sharp_A}}{2}= \left(\frac{T^{\sharp_A}T+TT^{\sharp_A}}{2}\right)^{\sharp_A}.
       \end{align*}
    Let us consider $x \in \mathcal{H} $ with $\|x\|_{A}=1.$ Then, we have
    	\begin{align*}
    		|\langle Tx,x\rangle_{A}|^2 &=  |\langle x,Tx\rangle_{A}|^2 \\ & = |\langle T^{\sharp_A}x,x\rangle_{A}|^2 \\ & = |\langle \left(\Re_{A} (T^{\sharp_A})+ i\Im_{A} (T^{\sharp_A})\right)x,x\rangle_{A}|^2 \\ & =  |\langle \Re_{A} (T^{\sharp_A})x,x\rangle_{A}|^2 +
    		|\langle \Im_{A} (T^{\sharp_A})x,x\rangle_{A}|^2  \\ & \geq \frac{1}{2}|\langle (\Re_{A} (T^{\sharp_A}) \pm \Im_{A} (T^{\sharp_A}))x,x\rangle_{A}|^2.
        \end{align*}
     Taking supremum over all $x$ in $\mathcal{H}$ with $\|x\|_{A}=1$, we get that
     \begin{align*}
     &	\sup_{\|x\|_{A}=1} \alpha |\langle Tx,x\rangle_{A}|^2 \\
     	 & \geq \sup_{\|x\|_{A}=1} \frac{ \alpha}{2}|\langle (\Re_{A} (T^{\sharp_A}) \pm \Im_{A} (T^{\sharp_A}))x,x\rangle_{A}|^2 \\
     	 & = \frac{ \alpha}{2} \left\| \Re_{A} (T^{\sharp_A}) \pm \Im_{A} (T^{\sharp_A})\right\|_A^2 \\ 
     	& = \frac{ \alpha}{2} \left\| (\Re_{A} (T^{\sharp_A}) \pm \Im_{A} (T^{\sharp_A}))^2\right\|_A,\,\,\textit{since}~~\Re_{A} (T^{\sharp_A}) \pm \Im_{A} (T^{\sharp_A})~~\textit{is A-self-adjoint} \\ 
     	& = \sup_{\|x\|_{A}=1} \langle \frac{ \alpha}{2} (\Re_{A} (T^{\sharp_A}) \pm \Im_{A} (T^{\sharp_A}))^2x,x\rangle_{A} \\ & \geq \sup_{\|x\|_{A}=1} \langle \frac{ \alpha}{4} \left((\Re_{A} (T^{\sharp_A}) + \Im_{A} (T^{\sharp_A}))^2+(\Re_{A} (T^{\sharp_A}) - \Im_{A} (T^{\sharp_A}))^2\right)x,x\rangle_{A} \\ & \geq \sup_{\|x\|_{A}=1} \langle \frac{ \alpha}{2} (\Re_{A}^2 (T^{\sharp_A}) + \Im_{A}^2 (T^{\sharp_A})^2)x,x\rangle_{A}.
     \end{align*}
     Hence, $$\left\|T\right\|_{A_\alpha}^2 \geq \sup_{\|x\|_{A}=1} \langle \frac{ \alpha}{2} (\Re_{A}^2 (T^{\sharp_A}) + \Im_{A}^2 (T^{\sharp_A}))x,x\rangle_{A}. $$
     Also we have that, $$\left\|T\right\|_{A_\alpha}^2 \geq \sup_{\|x\|_{A}=1} \langle (1-\alpha) T^{\sharp_A}Tx,x\rangle_{A}.$$
     Therefore,
     \begin{align*}
     	2\left\|T\right\|_{A_\alpha}^2  & \geq \sup_{\|x\|_{A}=1} \left\langle \left(\frac{ \alpha}{2} (\Re_{A}^2 (T^{\sharp_A}) + \Im_{A}^2 (T^{\sharp_A}))+(1-\alpha) T^{\sharp_A}T\right)x,x\right\rangle_{A} \\ & = \left\|\frac{ \alpha}{2} (\Re_{A}^2 (T^{\sharp_A}) + \Im_{A}^2 (T^{\sharp_A}))+(1-\alpha) T^{\sharp_A}T\right\|_{A} \\ & =  \left\|\frac{ \alpha}{2} \left(\frac{T^{\sharp_A}T+TT^{\sharp_A}}{2}\right)^{\sharp_A}+(1-\alpha) (T^{\sharp_A}T)^{\sharp_A}\right\|_{A} \\ & = \left\|\frac{\alpha}{4}(T^{\sharp_A}T+TT^{\sharp_A})+(1-\alpha) T^{\sharp_A}T\right\|_{A}. 
    \end{align*}
     Hence,
     \begin{align}\label{eqn5}
     	\left\|T\right\|_{A_\alpha}^2 \geq \frac{1}{2}\left\|\frac{\alpha}{4}(T^{\sharp_A}T+TT^{\sharp_A})+(1-\alpha) T^{\sharp_A}T\right\|_{A}.
     \end{align}
    Again, by using similar arguments as above, we have that
     \begin{align*}
     	3\left\|T\right\|_{A_\alpha}^2 & \geq  \sup_{\|x\|_{A}=1} \left\langle \left( \alpha (\Re_{A}^2 (T^{\sharp_A}) + \Im_{A}^2 (T^{\sharp_A}))+(1-\alpha) T^{\sharp_A}T\right)x,x\right\rangle_{A} \\ & = \left\|\frac{\alpha}{2}(T^{\sharp_A}T+TT^{\sharp_A})+(1-\alpha) T^{\sharp_A}T\right\|_{A}. 
     \end{align*}
    Hence,
    \begin{align}\label{eqn6}
    	\left\|T\right\|_{A_\alpha}^2 \geq \frac{1}{3}\left\|\frac{\alpha}{2}(T^{\sharp_A}T+TT^{\sharp_A})+(1-\alpha) T^{\sharp_A}T\right\|_{A}.
    \end{align}
    Combining (\ref{eqn5}) and (\ref{eqn6}) we get the inequality  (i).
    Now we prove the inequality  (ii). We have
     \begin{align*}
     	 |\langle Tx,x\rangle_{A}|^2 & = |\langle \Re_{A} (T^{\sharp_A})x,x\rangle_{A}|^2 +
     	|\langle \Im_{A} (T^{\sharp_A})x,x\rangle_{A}|^2 \\ & \leq \|\Re_{A} (T^{\sharp_A})x\|_A^2+\|\Im_{A} (T^{\sharp_A})x\|_A^2 \\ & =  \langle \Re_{A}^2 (T^{\sharp_A})x,x\rangle_{A} +
     	\langle \Im_{A}^2 (T^{\sharp_A})x,x\rangle_{A} \\ &=  \langle (\Re_{A}^2 (T^{\sharp_A})+\Im_{A}^2 (T^{\sharp_A}))x,x\rangle_{A}. 
     \end{align*}
     So, 
     \begin{align*}
     	 &\,\, \alpha|\langle Tx,x\rangle_{A}|^2+(1-\alpha)\|Tx\|_{A}^2 \\
     	 & \leq \langle \alpha(\Re_{A}^2 (T^{\sharp_A})+\Im_{A}^2 (T^{\sharp_A}))x,x\rangle_{A} +\langle (1-\alpha) T^{\sharp_A}Tx,x\rangle_{A} \\ & = \left\langle \left(\alpha (\Re_{A}^2 (T^{\sharp_A}) + \Im_{A}^2 (T^{\sharp_A}))+(1-\alpha) T^{\sharp_A}T \right)x,x\right\rangle_{A} \\ & \leq   \left\| \alpha (\Re_{A}^2 (T^{\sharp_A}) + \Im_{A}^2 (T^{\sharp_A}))+(1-\alpha) T^{\sharp_A}T\right\|_{A}  \\ & =  \left\| \alpha \left(\frac{T^{\sharp_A}T+TT^{\sharp_A}}{2}\right)^{\sharp_A}+(1-\alpha) (T^{\sharp_A}T)^{\sharp_A}\right\|_{A} \\ & = \left\|\frac{\alpha}{2}(T^{\sharp_A}T+TT^{\sharp_A})+(1-\alpha) T^{\sharp_A}T\right\|_{A}.
     \end{align*}
    Therefore, by taking supremum over all $x \in \mathcal{H}$ with $\|x\|_{A}=1,$ we get the inequality (ii).
     Now, using Theorem \ref{theo1} in the inequality (ii), we get $$ w_{A}^2(T) \leq \left\|\frac{\alpha}{2}(T^{\sharp_A}T+TT^{\sharp_A})+(1-\alpha) T^{\sharp_A}T\right\|_{A}.$$
     This holds for all $\alpha$, so
     considering minimum over $\alpha$, we get that
     $$w_{A}^2(T) \leq \min_{\alpha}  \left\|\frac{\alpha}{2}(T^{\sharp_A}T+TT^{\sharp_A})+(1-\alpha) T^{\sharp_A}T\right\|_{A}.$$
     This is the first inequality in (iii).
     The second inequality in (iii) follows from the case $\alpha=1.$
    \end{proof}
    
    \begin{remark}\label{pr}
    	Zamani in \cite[ Theorem 2.10 ]{zamani} proved that for $T \in \mathcal{B}_{A}(\mathcal{H}) $ , $$w_{A}^2(T) \leq \frac{1}{2}\|TT^{\sharp_A}+T^{\sharp_A}T\|_{A}.$$
    	Clearly, the first inequality obtained in Theorem \ref{theo5}(iii) improves on \cite[ Theorem 2.10 ]{zamani}. Here we consider an example to show proper improvement. Let $T=\left(\begin{matrix}
    	0&0&0\\
    	2&0&0\\
    	0&1&0
    	\end{matrix}  \right)$ and $A=\left(\begin{matrix}
    	1&0&0\\
    	0&1&0\\
    	0&0&2
    	\end{matrix}  \right).$ Then, simple calculations show that   $$ \left\|\frac{\alpha}{2}(T^{\sharp_A}T+TT^{\sharp_A})+(1-\alpha) T^{\sharp_A}T\right\|_{A} =\frac{23}{8} \,\, \, \text{for}\,\, \, \alpha=\frac{7}{8}$$ and $$  \frac{1}{2}\left\|T^{\sharp_A}T+TT^{\sharp_A}\right\|_{A}=3.$$ Hence, for the matrices $T$ and $A$,
    	$$ \min_{\alpha} \left\|\frac{\alpha}{2}(T^{\sharp_A}T+TT^{\sharp_A})+(1-\alpha) T^{\sharp_A}T\right\|_{A} < \frac{1}{2}\left\|T^{\sharp_A}T+TT^{\sharp_A}\right\|_{A}.$$

    \end{remark}
    
    To prove next theorem first we recall  the following lemma.
    
    \begin{lemma}$($\cite[Lemma 1]{ab}$)$\label{lemma2}
    	Let $a,b,e \in \mathcal{H} $ with $\|e\|_{A}=1.$ Then,
    	$$|\langle a,e \rangle_{A}\langle e,b \rangle_{A}| \leq \frac{1}{2} \left(|\langle a,b \rangle_{A}|+\|a\|_{A}\|b\|_{A}\right).$$
    \end{lemma}

    \begin{theorem}\label{theo6}
    		Let $T \in \mathcal{B}_{A}(\mathcal{H}).$ Then the following inequalities hold :
    		\begin{align*}
    		(i)~~ \left\|T\right\|_{A_\alpha}^2& \leq 	\frac{\alpha}{2}w_{A}(T^2)+\left\|\frac{\alpha}{4}(T^{\sharp_A}T+TT^{\sharp_A})+(1-\alpha) T^{\sharp_A}T\right\|_{A}\\ 
    		(ii)~~ w_{A}^2(T) \leq & \min_{\alpha} \left\{	\frac{\alpha}{2}w_{A}(T^2)+\left\|\frac{\alpha}{4}(T^{\sharp_A}T+TT^{\sharp_A})+(1-\alpha) T^{\sharp_A}T\right\|_{A}\right\}\\ \leq & \frac{1}{2}w_{A}(T^2)+\frac{1}{4}\|T^{\sharp_A}T+TT^{\sharp_A}\|_A.
    		\end{align*}
    \end{theorem}
    
    \begin{proof}
    	Let $x \in \mathcal{H} $ with $\|x\|_{A}=1.$ Taking $a=Tx,$ $b=T^{\sharp_A}x$ and $e=x$ in Lemma \ref{lemma2}, and using the arithmetic-geometric mean inequality, we get
    	\begin{align*}
    		|\langle Tx,x\rangle_{A}|^2 & \leq \frac{1}{2} \left( |\langle T^2x,x\rangle_{A}| + \| Tx\|_{A} \| T^{\sharp_A}x\|_{A} \right)\\ & \leq \frac{1}{2}|\langle T^2x,x\rangle_{A}| +\frac{1}{2}\langle T^{\sharp_A}Tx,x\rangle_{A}^{1/2}\langle TT^{\sharp_A}x,x\rangle_{A}^{1/2}\\ & \leq \frac{1}{2}|\langle T^2x,x\rangle_{A}|+ \frac{1}{4}\left(  \langle T^{\sharp_A}Tx,x\rangle_{A}+\langle TT^{\sharp_A}x,x\rangle_{A}\right) \\ & = \frac{1}{2}|\langle T^2x,x\rangle_{A}|+\frac{1}{4}\langle (T^{\sharp_A}T+TT^{\sharp_A})x,x\rangle_{A}.
    	\end{align*}
    Therefore,
        \begin{align*}
        	& \alpha|\langle Tx,x\rangle_{A}|^2+(1-\alpha) \|Tx\|_{A}^2 \\ & \leq  \frac{\alpha}{2}|\langle T^2x,x\rangle_{A}|+\left\langle  \left(\frac{\alpha}{4}(T^{\sharp_A}T+TT^{\sharp_A})+(1-\alpha) T^{\sharp_A}T \right)x,x\right\rangle_{A} \\ & \leq \frac{\alpha}{2}w_{A}(T^2)+\left\|\frac{\alpha}{4}(T^{\sharp_A}T+TT^{\sharp_A})+(1-\alpha) T^{\sharp_A}T\right\|_{A}.
        \end{align*}
     Taking supremum over all $x$ in $\mathcal{H}$ with $\|x\|_{A}=1,$ we get
     $$\left\|T\right\|_{A_\alpha}^2 \leq 	\frac{\alpha}{2}w_{A}(T^2)+\left\|\frac{\alpha}{4}(T^{\sharp_A}T+TT^{\sharp_A})+(1-\alpha) T^{\sharp_A}T\right\|_{A}.$$
     This is the inequality in (i).
      Now using Theorem \ref{theo1} in (i), we get
      $$ w_{A}^2(T) \leq 	\frac{\alpha}{2}w_{A}(T^2)+\left\|\frac{\alpha}{4}(T^{\sharp_A}T+TT^{\sharp_A})+(1-\alpha) T^{\sharp_A}T\right\|_{A},$$
      This holds for all $\alpha$, so
      taking infimum over all $\alpha$, we get the first inequality in (ii). The remaining inequality follows from the case $\alpha=1.$
    \end{proof}

It should be mentioned here that the inequality, $w_{A}^2(T) \leq  \frac{1}{2}w_{A}(T^2)+\frac{1}{4}\|T^{\sharp_A}T+TT^{\sharp_A}\|_A$ is known and it was given in \cite[Th. 2.11]{zamani}. Clearly, the first inequality in Theorem \ref{theo6} (ii) is sharper than the above mentioned inequality. To show that the  inequality in Theorem \ref{theo6} (ii)  is strictly sharper, we consider the following example.

 \begin{example}
 	
 Let $T=\left(\begin{matrix}
 0&0&0\\
 2&0&0\\
 0&1&0
 \end{matrix}  \right)$ and $A=\left(\begin{matrix}
 1&0&0\\
 0&1&0\\
 0&0&2
 \end{matrix}  \right).$ Then, simple calculations show that   $$ \frac{\alpha}{2}w_{A}(T^2)+\left\|\frac{\alpha}{4}(T^{\sharp_A}T+TT^{\sharp_A})+(1-\alpha) T^{\sharp_A}T\right\|_{A} =\frac{6\sqrt{2}+20}{13}\approx 2.19117549033 $$  \text{for} $\alpha=\frac{12}{13}$ and $$  \frac{1}{2}w_{A}(T^2)+\frac{1}{4}\|T^{\sharp_A}T+TT^{\sharp_A}\|_A= \frac{3+\sqrt{2} } {2} \approx 2.20710678119.$$ Hence, for the matrices $T$ and $A$,
 \begin{eqnarray*}
  && \min_{\alpha} \left\{	\frac{\alpha}{2}w_{A}(T^2)+\left\|\frac{\alpha}{4}(T^{\sharp_A}T+TT^{\sharp_A})+(1-\alpha) T^{\sharp_A}T\right\|_{A}\right\}\\
  & < & \frac{1}{2}w_{A}(T^2)+\frac{1}{4}\|T^{\sharp_A}T+TT^{\sharp_A}\|_A.
 \end{eqnarray*}
 \end{example}

    Finally, we prove the following inequalities.
    
    \begin{theorem}\label{theo8}
    	Let $T \in \mathcal{B}_{A}(\mathcal{H}).$ Then
    	\begin{align*}
    	& (i)~~ \left\|T\right\|_{A_\alpha}^2 \leq \|(1-\alpha) T^{\sharp_A}T+ \alpha  TT^{\sharp_A}\|_{A},\\
    	& (ii)~~w_A^2(T) \leq \min_{\alpha} \| \alpha T^{\sharp_A}T+(1-\alpha)TT^{\sharp_A}\|_{A}.
    	\end{align*}
    \end{theorem}

    \begin{proof}
    	Let $x \in \mathcal{H} $ with $\|x\|_{A}=1.$ Then we get that
    	\begin{align*}
    		 \alpha |\langle Tx,x \rangle_{A}|^2+(1-\alpha) \| Tx\|_{A}^2  \leq &   (1-\alpha) \|Tx\|_{A}^2+\alpha \|T^{\sharp_A}x\|_{A}^2 \\ = &  (1-\alpha) \langle Tx,Tx \rangle_{A}+\alpha  \langle T^{\sharp_A}x,T^{\sharp_A}x \rangle_{A}\\ = & (1-\alpha)\langle T^{\sharp_A}Tx,x \rangle_{A}+\alpha \langle TT^{\sharp_A}x,x \rangle_{A}\\
    		 \leq & \| (1-\alpha) T^{\sharp_A}T+ \alpha TT^{\sharp_A}\|_{A}.
    	\end{align*}
     Taking supremum over all $x$ in $\mathcal{H}$ with $\|x\|_{A}=1,$ we get
     $$\left\|T\right\|_{A_\alpha}^2 \leq \|(1-\alpha) T^{\sharp_A}T+ \alpha TT^{\sharp_A}\|_{A},$$  Now, by using Theorem \ref{theo1} in (i) and considering minimum over $\alpha$, we get the inequality in (ii). 
    \end{proof}

    Here we note that the inequality, $w_A^2(T) \leq \min_{\mu \in [0,1]}\| \mu T^{\sharp_A}T+(1-\mu)TT^{\sharp_A}\|_{A},$ recently obtained in \cite[Th. 1]{P1}, is same as obtain in (ii) of Theorem \ref{theo8}.

  We conclude that by applying the $A_\alpha$-semi-norm of operators in $\mathcal{B}_{A}(\mathcal{H})$ we always get the A-numerical radius inequalities of operators in $\mathcal{B}_{A}(\mathcal{H})$, which  improve on the existing ones.

\bibliographystyle{amsplain}

\begin{thebibliography}{99}
		
	\bibitem{acg1}{M.L. Arias, G. Corach and M.C. Gonzalez,} {Partial isometries in semi-Hilbertian spaces,} Linear Algebra Appl. 428(7) (2008) 1460-1475.
	
	
	\bibitem{arias 1} M.L. Arias, G. Corach, M.C. Gonzalez, Lifting properties in operator ranges, Acta Sci. Math. (Szeged) 75:3-4 (2009) 635-653.
	
   \bibitem{bak}  H. Baklouti, K. Feki, O.A.M. Sid Ahmed, Joint numerical ranges of operators in semi-Hilbertian spaces, Linear Algebra Appl. 555 (2018) 266-284.
   
   
   \bibitem{P1} P. Bhunia, R.K. Nayak and K. Paul, Improvement of A-numerical radius inequalities of semi-Hilbertian space operators, {Results Math.} 76, 120 (2021). \url{https://doi.org/10.1007/s00025-021-01439-w} 
	
	\bibitem{3}	P. Bhunia, S. Bag, K. Paul, Numerical radius inequalities and its application	in estimation of zeros of polynomials, Linear Algebra Appl. 573 (2019) 166-177.
	 
	 \bibitem{1} P. Bhunia, K.Paul, Proper improvement of well-known numerical radius inequalities and their applications. https://arxiv.org/abs/2009.03206
	
	\bibitem{2} P. Bhunia, K. Paul, Some improvements of numerical radius inequalities of operators and operator matrices, Linear Multilinear Algebra, (2020).
\url{	https://doi.org/10.1080/03081087.2020.1781037 }
	
	\bibitem{a1}  P. Bhunia, K. Paul, R.K. Nayak, On inequalities for A-numerical radius of operators, Electron. J. Linear Algebra 36 (2020) 143-157.
	
	\bibitem{a2}  P. Bhunia, R.K. Nayak, K. Paul, Refinements of A-numerical radius inequalities and their applications, Adv. Oper. Theory 5 (2020) 1498-1511.
	
    \bibitem{ab}  A. Bhanja,  P. Bhunia, K. Paul,  On generalized Davis–Wielandt radius inequalities of semi-Hilbertian space operators, Oper. Matrices (2021), to appear.
    
    \bibitem{BBSP} P. Bhunia, A. Bhanja, D. Sain and K. Paul, Numerical radius inequalities of operator matrices from a new norm on $\mathcal{B}(\mathcal{H})$, (2021). \url{http://arxiv.org/abs/2008.13708}
	
    \bibitem{dog} R.G. Douglas, On majorization, factorization and range inclusion of operators in Hilbert space, Proc. Amer. Math. Soc. 17 (1966) 413-416.
	
	\bibitem{feki 1} K. Feki, Spectral radius of semi-Hilbertian space operators and its applications, Ann. Funct. Anal. 11 (2020) 929-946.
	
	\bibitem{feki 2}   K. Feki, A note on the A-numerical radius of operators in semi-Hilbert spaces, Arch. Math. 115 (2020) 535-544.
	
	\bibitem{4} F. Kittaneh, Numerical radius inequalities for Hilbert spaces operators, Studia	Math. 168 (2005) 73-80.
	
	\bibitem{ms} M.S. Moslehian, M. Kian, Q. Xu, Positivity of $2 \times 2$ block matrices of operators, Banach J. Math. Anal. 13 (2019) 726-743.
	
	\bibitem{mose} M.S. Moslehian, Q. Xu, A. Zamani, Seminorm and numerical radius inequalities of operators in semi-Hilbertian spaces, Linear Algebra Appl. 591 (2020) 299-321.
	
    \bibitem{SBBP} D. Sain,  P. Bhunia,  A. Bhanja, K. Paul, On a new norm on $\mathcal{B}(\mathcal{H})$ and its applications to numerical radius inequalities, Ann. Funct. Anal. 12, 51 (2021). \url{https://doi.org/10.1007/s43034-021-00138-5} 
	
    \bibitem{zama}	A. Zamani, Some lower bounds for the numerical radius of Hilbert space operators, Adv. Oper. Theory 2 (2017) 98-107.
	
    \bibitem{zamani} A. Zamani, $A$-numerical radius inequalities for semi-Hilbertian space operators, Linear Algebra Appl. 578 (2019) 159-183.
	
	
\end{thebibliography}

\end{document}